\documentclass[11pt]{article}
\usepackage[leqno]{amsmath}
\usepackage{amssymb}
\usepackage{graphicx,amssymb} 
\usepackage[small]{caption}
\usepackage{amssymb,amsfonts}
\usepackage[backrefs,lite,alphabetic]{amsrefs} %bibliography
\usepackage[all,arc]{xy}
\usepackage{cancel}
\usepackage{enumerate}
\usepackage{mathrsfs}
\usepackage{pgfplots}
\usepackage{fullpage}
\usepackage{changepage}% http://ctan.org/pkg/changepage
\usepackage{amsthm}

\pgfplotsset{soldot/.style={color=blue,only marks,mark=*}} \pgfplotsset{holdot/.style={color=blue,fill=white,only marks,mark=*}}
\newcommand{\beqn}{\begin{eqnarray*}}
\newcommand{\eeqn}{\end{eqnarray*}}

\numberwithin{equation}{section}
\newtheorem{lemma}[equation]{Lemma}

\newtheorem{prop}[equation]{Proposition}
\newtheorem{cor}[equation]{Corollary}

\newtheorem{claim*}{Claim}
\newtheorem{thm}[equation]{Theorem}

\theoremstyle{definition}
\newtheorem{rmk}[equation]{Remark}

\newtheorem{eg}[equation]{Example}

\newtheorem{defn}[equation]{Definition}

\newcommand{\Z}{\mathbb{Z}}

\newcommand{\C}{\mathbb{C}}

\title{Explicit Bound on Collective Strength of Regular Sequences of Three Homogeneous Polynomials}
\author{Chiahui Cheng\\
Department of Mathematics\\
University of Wisconsin, Madison\\
USA}

\begin{document}
\maketitle

\begin{abstract}
Let $f_1,\cdots,f_r\in k[x_1,\cdots,x_n]$ be homogeneous polynomial of degree $d$. Ananyan and Hochster (2016) proved that there exists a bound $N=N(r,d)$ where if collective strength of $f_1,\cdots,f_r\geq N$, then $f_1,\cdots,f_r$ are regular sequence. In this paper, we study the explicit bound $N(r,d)$ when $r=3$ and $d=2,3$ and show that $N(3,2)=2$ and $N(3,3)>2$.\\
\end{abstract}

\begin{center}
\large{
\textbf{Contents}\\\
}
\end{center}

1. Introduction\hfill 2\\
2. Notation and background\hfill 3\\
3. The bound on collective strength when $r=1$ and $r=2$\hfill 4\\
4. Minors of the $(n+1)\times n$ matrix\hfill 5\\
5. The bound on collective strength when $r=3$ and $d=2$\hfill 5\\
6. Homogeneous polynomials and column grading\hfill 9\\
7. The bound on collective strength when $r=3$ and $d=3$\hfill 10\\
8. References\hfill 18\\

\newpage
\setcounter{section}{0}

% SECTION 1

\section{Introduction}

Let $R$ denote a polynomial ring over an algebraic closed field $K$, say $R=K[x_1,\cdots,x_N]$. Let $I$ be an ideal of $R$ generated by $r$ homogeneous polynomials of degree at most $d$. Stillman's Conjecture states that there is a bound $M(r,d)$ on the projective dimension of $I$, which depends only on $r$ and $d$.\\

As part of their proof, Ananyan and Hochster \cite{Ananyan-Hochster-small} introduced the term strength $k$ or $k$-strong, which is equivalent to the following definition:

\begin{defn}\label{strength} Let $f$ be a homogeneous element of $R$. The \textbf{strength} of $f$ is the minimal integer $k\geq 0$ for which there is a decomposition $f =\sum_{i=1}^{k+1} g_ih_i$ with $g_i$ and $h_i$ homogeneous elements of $R$ of positive degree, or $\infty$ if no such decomposition exists. We take the convention that the strength of any constant, including 0, is -1. The \textbf{collective strength} of a set of homogeneous elements $\{f_i\}_{i\in I}$ of $R$ is the minimal strength of a non-trivial homogeneous $k$-linear combination.\\
\end{defn}

%
%Let $K$ and $R$ be defined as above. Let V denote a graded $K$-vector subspace of $R$ of vector space dimension $n$ with dimension sequence $(\delta_1, . . . , \delta_d)$, such that for $1\leq i \leq d$, in [2] the authors proved that if the strength of every nonzero element of $V_i$ is at least $A_i(\delta)$ (respectively, $^\eta A_i(\delta) )$, then every sequence of $K$-linearly independent forms in $V$ is a regular sequence (respectively, is an $R_\eta$-sequence). And they used this to prove the Stillman's conjecture:\\
%
%\textbf{Stillman's Conjecture} Let $R=k[x_1,\cdot,x_n]$. Given a specific number $r$ of homogeneous polynomials of specific positive degrees, say at most $d$, there is a bound for the projective dimension of the ideal $I$ the homogeneous polynomials generate that depends on $n$ and $d$ but not the number $n$ of variables. i.e. let $I=(f_1,\cdots,f_r)$, we have pdim$(I)\leq M(r,d)$.\\

Ananyan and Hochster \cite{Ananyan-Hochster-strength} proved that there is a bound on the projective dimension of $R/$I that depends only on $n$, and not on $N$. They also showed that there are functions $^\eta A(n, d)$ with the following property: in a graded $n$-dimensional $K$-vector subspace $V$ of $R$ spanned by forms of degree at most $d$, if no nonzero form in $V$ is in an ideal generated by $^\eta A(n, d)$ forms of strictly lower degree, which is the strength condition, then any homogeneous basis for $V$ is an regular sequence.\\

In later work, \cite{Ananyan-Hochster-strength}, they obtained the function $^\eta A(n,d)$ in degree 2, 3 and 4. In degree 2,  they obtained an explicit value for $^\eta B(n,2)$ that gives the Stillman’s conjecture for quadrics when there is no restriction on $n$. In particular, for an ideal $I$ generated by $n$ quadrics, the projective dimension $R/I$ is at most $2^{n+1}(n-2) + 4$.\\

On the other hand, Erman, Sam and Snowden \cite{Erman-Sam-Snowden} used an alternative approach, i.e., proving that the limiting rings are polynomial rings and using these polynomiality results to show the existence of small subalgebras $^\eta \mathcal{B}(n,d)$ in \cite{Ananyan-Hochster-small} and Stillman's conjecture.\\

%
%\textbf{Theorem 1} ([3]) Given positive integers $d_1,\cdots,d_r$ there exists an integer $N=N(d_1,\cdots, d_r)$ with the following property: If $f_1,\cdots,f_r\in k[x_1,\cdots,x_n]$ are homogeneous polynomial of degree $d_1,\cdots,d_r$ and collective strength at least $N$, then $f_1,\cdots,f_r$ form a regular sequence.\\
%
%Having this, they then proved Stillman's conjecture in the following theorem:\\
%
%\textbf{Theorem 2} ([3])  Given positive integers $d_1,\cdots,d_r$ there exists an integer $s=s(d_1,\cdots,d_r)$ with the following property. If $f_1,\cdots,f_r$ are homogeneous elements of $k[x_1,\cdots,x_n]$, for any infinite perfect field $k$ and any $n$, with deg$(f_i)=d_i$, then
%\begin{enumerate}[(a)]
%\item There exists a regular sequence $g_1,\cdots,g_s$ in $k[x_1,\cdots,x_n]$, where each $g_i$ is homogeneous of degree at most $\max (d_1,\cdots,d_r)$, such that $f_1,\cdots,f_r$ are contained in the subalgebra $k[g_1,\cdots,g_s]$.
%\item The ideal $(f_1,\cdots,f_r)$ has projective dimension at most $s$.
%\end{enumerate}

In fact, in \cite{Ananyan-Hochster} when proving that there is a bound on the projective dimension of $R/I$ where $R=(x_1,\cdots,x_N)$ and $I$ is an ideal of $R$ generated by $n$ polynomials of degree at most 2, that depends only on $n$, not on $N$, Ananyan and Hochster show that there exists a positive integer $C(n)$ for positive integer $n$ such that these $n$ homogeneous polynomials is a regular sequence when their collective strength $\leq C(n)$ and $C(n)$ approaches to $2n^{2n}$.\\

With all these as the motivation, we are interested in finding the explicit bound $N=N(r,d)$. i.e., For homogeneous polynomials $f_1,\cdots,f_r$ of degree $\leq d$, if $N=N(r,d)$ is the bound for the collective strength such that $\{f_i\}$ is a regular sequence, we would like to know that if $N$ depends on (increase with) $d$ ?\\

More specifically, we study the explicit bound $N=N(r,d)$ on collective strength of regular sequence of three homogeneous polynomials, $f_1,f_2,f_3$. Note that $r=3$ is the first interesting case. In fact, we will give a simple proof that $N$ does not depend on $d$ when $r=1$ or 2. Our main results will show that $N(3,2)=2$ and $N(3,3)>2$.\\

%In our paper, for $r=3$ and $d=2$, we define minrank$(f_1,f_2)$ to be the lowest rank of a linear combination of $f_1,f_2$ and prove that minrank = codim $J$ where $J$ is the Jacobian matrix of $f_1,f_2$, then use this to prove that $N(3,2)\leq 2$. For the other direction, we introduce the minors of the $(n+1) \times n$ matrix and used it to prove that $N(3,2)\geq 2$, hence we obtained $N(3,2)=2$, which means, the bound for the collective strength is exactly 2.\\
%
%For $r=3$ and $d=3$, we first introduce column grading, then take advantage of the symmetry of the minors of $(n+1)\times n$ and use the property of homogeneous polynomials in column grading to prove that there exists cubic homogeneous polynomials $f_1,f_2,f_3$ with collective strength 3 that is not a regular sequence. i.e. $N(3,3)>2$.\\

% SECTION 2

\section{Notation and Background}

Here are some basic foundations, notations and background that are used in this paper:

\begin{defn} A sequence of elements $x_1,\cdots,x_d$ in a ring $R$ is called a \textbf{regular sequence} on $R$ if the ideal $\langle x_1,\cdots,x_i\rangle$ is proper and for each $i$ with $i=1,\cdots,d-1$, the image of $x_{i+1}$ is a nonzerodivisor in $R/\langle x_1,\cdots,x_i\rangle$.
\end{defn}

%\begin{defn} If $I$ is a prime ideal, the \textbf{codimension of $I$} is the dimension of the local ring $R_I$. Equivalently, it is the supremum of lengths of chains of primes descending from $I$. If $I$ is not prime, then we define \textbf{codim $I$} to be the minimum of the codimensions of the primes containing $I$. Codimension in algebriac geometry is defined by Codim $(I)$ = Codim $(V(I)\subseteq \mathbb{A}^n)$.
%\end{defn}

%\textbf{Definition} A prime ideal $P$ of $R$ is said to be associated to the $R$-module $M$ if $P$ is the annihilated of some element of $M$.\\

\begin{thm}\textbf{Principal Ideal Theorem}\label{PIT} \cite[Theorem 10.1]{Eisenbud} If $x\in R$, and $P$ is the minimum among primes of $R$ containing $(x)$, then codim $P\leq 1$.
\end{thm}

\begin{defn} \cite{Miller-Stephenson} \label{quadratic} A \textbf{quadratic form} is a polynomial in several variables with coefficients in an algebraically closed filed $\mathbf{C}$, all of whose terms have degree 2. If the variables in question are $x,y,$ and $z$, examples of quadratic forms include $x^2, xy-z^2$ and $2x^2+3y^2-5z^2$. To every quadratic form $q$ there is a nonnegative integer called the \textbf{rank} of $q$. For example, the quadratic form $x^2$ has rank 1. Quadratic forms of the form $\alpha x^2+\beta y^2$, where $\alpha,\beta\in\mathbf{C}^*$, have rank 2. In fact, every nonzero quadratic form $q$ can be written in the form $X_1^2+\cdots+X_k^2$ for some positive integer $k$ and some linearly independent set $\{X_1,\cdots,X_k\}$ of linear forms in the original variables, and $k$ is precisely the rank of $q$.
\end{defn}

\begin{rmk}\label{rank_strength} By changing the coordinate, one can see that any quadratic form of rank $k$ with $k$ even has strength $\frac{1}{2}k-1$ since $X_1^2+\cdots+X_k^2$ can be written as $X_1X_2+\cdots+X_{k-1}X_{k}$ that has $\frac{1}{2}k$ terms.
\end{rmk}

\begin{lemma}\label{dim} Let $(R,m)$ be a local ring and let $x\in m$ and assume that $x$ is not a zero divisor. Then dim $(R/(x))=$ dim $R-1$.
\end{lemma}

\begin{proof} Recall that dimension of $R/(x)$ is the maximum length of the chain $$(x)\subsetneq P_1\subsetneq P_2\subsetneq\cdots$$
since the set of zero divisors is the union of the associated primes $\cup_{\textbf{Assoc}(R)}P$. So if $x$ is not a zero divisor, $(x)$ does not belong to any associated prime. In particular, it does not belong to any minimum prime. So if we take a chain of prime ideals in $R$ containing $x$, say, taking a maximum chain, the smallest element we called the minimum prime. So if $x$ is not a zero divisor, $(x)$ is not in this chain. That means $(x)$ can never be the minimum prime in a chain. Hence, if $x$ is not a zero divisor, the chain containing $(x)$ has to be strictly shorter which means dim $R/(x)<$ dim $R$.\\

Now, Theorem \ref{PIT} (Principal Ideal Theorem) gives dim $R/(x)\geq$ dim$R$ - 1, hence we have the equality. In another words, adding the hypothesis of $x$ being a nonzero divisor drops dimension exactly one, i.e. $$\text{dim } R/(x) = \text{dim }R - 1$$
\end{proof}

\begin{prop}\label{reg} In a polynomial ring, $f_1,\cdots,f_r$ are regular sequence if and only if $I=\langle f_1,\cdots,f_r\rangle$ has codim $(I)=r$.
\end{prop}

\begin{proof}
If $f_1,\cdots,f_c$ regular sequence, then by applying Lemma \ref{dim} $c$ times, we get 

$$\text{dim} R/(f_1,\cdots,f_c) = \text{dim }R - c$$

which means
$$\text{codim } (f_1,\cdots,f_c)=c$$

About the other direction, note that Cohen-Macaulay ring is a ring where every system of parameters is automatically a regular sequence. So if codim $(f_1,\cdots,f_c)=c$, then $f_1,\cdots,f_c$ is a regular sequence.
\end{proof}

% SECTION 3

\section{The Bound on Collective Strength when $r=1$ and $r=2$}

We want to examine the explicit bound $N=N(r,d)$ for the collective strength of homogeneous polynomials $f_1,\cdots,f_r$ of degree $\leq d$ such that $\{f_i\}$ is a regular sequence. We would like to know whether $N(r,d)$ increases with $r$ fixed and $d\to\infty$.\\ 

When $r=1$, the answer is no since one polynomial is a regular sequence if and only if it is non-zero. In particular, $N(1,d)=0$. If $r=2$, we prove that $N(2,d)=1$ as below.

% Lemma 3.1%
\begin{lemma}\label{notreq} $f_1$ and $f_2$ are not a regular sequence if and only if gcd$(f_1,f_2)\neq 1$.
\end{lemma}

\begin{proof}
Let $I=(f_1,f_2)$. If  $f_1$ and $f_2$ are not a regular sequence, by Theorem \ref{PIT} Principal Ideal Theorem, we have codim $(I)<2$. If codim $(I)=0$, there's only one ideal that has codimension zero which is the zero ideal. But this would mean that $f_1=f_2=0$. So codim $(I)=0$ is the same as being the zero ideal. So if two polynomial which are the zero ideal, they both have to be zero. If codim $(I)=1$, by Primary Decomposition Theorem, the ideal $I$ can be decomposed into the intersection of primary ideals. And there exists codimension-one pieces. These codimension-one pieces will be principal and are generated by single element, called $(a)$. i.e. $(f_1,f_2)=(a)\cap\cdots$. What is the relationship between $a$ and $f_i$? The claim is $a\mid f_1$ and $a\mid f_2$. So we want to prove that $a$ is a common factor of $f_1$ and $f_2$. But according to property of intersection of ideals and definition of Primary Decomposition, this $a$ is actually the common factor of $f_1$ and $f_2$. This shows that $a$ divides the gcd$(f_1,f_2)$  and therefore gcd $(f_1,f_2)\neq 1$.\\

%The other direction is basically the same idea. If the gcd is not 1? say $f_1=gh_1$ and $f_2=gh_2$, i.e. $g$ is a common factor of $f_1$ and $f_2$. i.e. $(f_1,f_2)\subseteq (g)$ which means $V(f_1,f_2)\supseteq V(g)$ that vanishing locus contain a hypersurface. Every hypersurface is defined by a single polynomial and vice versa. So hypersurface is a priori of codimension one. In the polynomial ring is a Unique Factorization Domain. So there's a bijection between irreducible hypersurfaces and prime principal ideals (generated by one element). So every codimension-one ideal has this form, but that's not true for every ring.\\
%
%So now in a polynomial ring, every prime ideal of codimension one is principal. \\
%The Primary Decomposition of $I=\langle f_1. f_2\rangle$ will have a primary ideal $(g)$ of codimension-one. In this case, codimension of $I$ can not be 2, hence has to be less that 2. Therefore, $f_1,f_2$ are not a regular sequence.\\

On the other hand, if $f_1=gh_1$ and $f_2=gh_2$, need to show that $f_2$ is a zero divisor in $R/(f_1)=R/(gh_1)$. Since $h_1f_2=h_1(gh_2)=(h_1g)h_2=0$ on $R/(f_1)$. Hence $f_2$ is a zero divisor in $R/(f_1)$, so $f_1,f_2$ are not a regular sequence. 
\end{proof}

% Corollary 3.2%
\begin{cor} If strength$(f_1,f_2)\geq 1$, then $f_1, f_2$ are regular sequence.
\end{cor}

\begin{proof}
Note that an element is irreducible if and only if its strength is at least 1. By Lemma \ref{notreq}, if $f_1$, $f_2$ fail to be a regular sequence, then both elements must be reducible and have a common factor. So their collective strength is $\leq 0$.  Thus, if either $f_1$ or $f_2$ is irreducible, then they will automatically form a regular sequence.
\end{proof}

Therefore, we conclude that $N(2,d)=1$, i.e. $N$ does not depend on $d$.\\

% SECTION 4

\section{Minors of the $(n+1) \times n$ Matrix}

Let $\Phi$ be a $(n+1)\times n$ matrix $(x_{i,j})$ where $(x_{i,j})$ are variables in a polynomial rings $\C\C[x_{1,1},\cdots,x_{n+1,n}]$. Let $f_i$ be the $n\times n$ minor obtained by dropping the $i^{th}$ row. Then $\langle f_1,f_2,\cdots,f_{n+1}\rangle$ has codim 2 by the Hilbert-Burch Theorem \cite[Theorem 20.15]{Eisenbud}. Note that every $n\times n$ minor has strength $\leq n-1$ by Laplace expansion. \\

%A generic $n\times n$ determinant minor has strength $=n$. i.e. If $(x_{ij})$ is an $n\times n$ matrix of independent variables, then det$(x_{ij})$ has strength $n$.\\

Hence, if $f_1,f_2,f_3$ are $n\times n$ minors of the $(n+1) \times n$ matrix, by Proposition \ref{reg} they are not a regular sequence because the maximal minors of this matrix have codimension 2. i.e. $\langle f_1,f_2,f_3\rangle$ has codimension 2.\\

%The following can be omitted.\\
%
%{\color{red}
%
%\textbf{Question}: What is collective strength of $f_1,f_2,f_3$?\\

%\textbf{Lemma} Collectives strength of $f_1,f_2,f_3$ is strength of $f_1$.\\
%
%Proof\\
%
%This is given by the following equality.
%$$\lambda_1 \det f_1+\lambda \det f_2+\lambda \det f_3= \det
%\begin{bmatrix}
%\lambda_1 & -\lambda_2 & \lambda_3\\
%a & b & c\\
%d & e & f\\
%\end{bmatrix}
%$$
%
%this says that the collective strength is same as the strength of the single determinant. $\Box$\\
%
%So we don't need to do the collective strength.\\
%}

This is a natural way to build a collection of three homogeneous polynomials of degree $n$ that are not a regular sequence. To prove that $N(3,n)>n-1$, it is enough to show that their collective strength is $n-1$. Thus, there exists homogeneous polynomials $f_1,f_2,f_3$ of degree $n$ with collective strength $n-1$ that is not a regular sequence. i.e. $N(3,n)>n-1$.\\

% SECTION 5

\section{The Bound on Collective Strength when $r=3$ and $d=2$}

While the main goal of this paper is to prove that $N(3,3)>3$, this section will focus on the case $r=3$ and $d=2$.  We obtain a stronger result: $N(3,2) = 2$.

\begin{lemma}\label{prime} Let $f_1$, $f_2\in R$ be homogeneous polynomials. If the singular locus of $V(f_1,f_2)$ has codimension $>4$, then the ideal $\langle f_1,f_2\rangle$ is prime.
\end{lemma}

\begin{proof}
Let $X=V(f_1,f_2)$ be the vanishing locus of the ideal $\langle 
f_1,f_2\rangle$. If the ideal is not prime, the corresponding vanishing locus has multiple components. 
Assume for contradiction that the ideal $\langle f_1,f_2\rangle$ is not prime.  Then X will have multiple irreducible components, $X_1$ and $X_2$, each of which has codimension at two.  Since $X_1$ and $X_2$ are subsets of projective space, they will intersect in codimension at most 4, and this intersection will belong to the singular locus of $X$.  It follows that the singular locus of $X$ has codimension at most 4 inside of projective space. Hence if the singular locus of $X$ has codimension bigger than 4, then we have a contradiction, implying that the ideal $\langle f_1,f_2\rangle$ has to be prime.
\end{proof}

\begin{defn} For any two quadrics $f_1, f_2$ in $k[x_1,\cdots,x_n]$, define \textbf{minrank$(f_1,f_2)$} to be the lowest rank of a nontrivial $k$-linear combination of $f_1$ and $f_2$. Note that \textbf{minrank$(f_1,f_2)=0$} if and only if $f_1$ and $f_2$ are $k$-linear dependent.
\end{defn}

\begin{thm} \label{n32} If $f_1,f_2,f_3\in k[x_1,\cdots,x_n]$ are polynomials of degree 2 and collective strength $\geq 2$, then they form a regular sequence. i.e. $N(3,2)\leq 2$.
\end{thm}

\begin{proof}
Since we are over an algebraically closed field of characteristic not equal to 2, \cite{Harris} Lemma 22.42 implies that we can simultaneously write $f_1, f_2$ in $k[x_1,\cdots,x_n]$ as $$f_1=a_1x_1^2+a_2x_2^2+\cdots+a_nx_n^2$$ $$f_2=b_1x_1^2+b_2x_2^2+\cdots+b_nx_n^2$$ 

It is enough to show that $\langle f_1,f_2\rangle$ is prime. This is because, if $\langle f_1,f_2\rangle$ is prime then $k[x_1, \cdots, x_n]/\langle f_1,f_2\rangle$ will have no zero divisors, and so $f_3$ will be a non-zero divisor modulo $\langle f_1,f_2\rangle$.  Hence, $f_1, f_2, f_3$ would be a regular sequence. Thus, our goal will be to show that the hypothesis on collective strength implies that $\langle f_1,f_2\rangle$ is a prime ideal.\\

Recall that the Jacobian ideal of $\langle f_1,f_2\rangle$ is generated by the $2\times 2$ minors of the Jacobian matrix of $\langle f_1,f_2\rangle$. It describes the singular locus. Now, if $a_i=b_i=0$ for some $i$, then we can simply reduce the number of variables.  So we can assume that either $a_i$ or $b_i$ is nonzero for all $i$.  Moreover, by replacing $f_1$ by an appropriate $k$-linear combination of $f_1$ and $f_2$, we can further assume that $a_i \neq 0$ for all $i$.  And finally, since we are over an algebraically closed field, we can rescale the $x_i$-coordinates so that $f_1=x_1^2+\cdots+x_n^2$ and $f_2=b_1x_1^2+\cdots+b_nx_n^2$. The Jacobian matrix is

\[\begin{bmatrix}
2x_1 & 2x_2 & \cdots & 2x_n\\
2b_1x_1& 2b_2x_2 & \cdots & 2b_nx_n\\
\end{bmatrix}.
\]

Let Jac$(f_1,f_2)=\langle 2\times 2 \text{ minors of Jacobian matrix} \rangle$ denote the Jacobian ideal of $f_1, f_2$. So $(i,j)$ minor $=(4b_j-4b_i)x_ix_j$. For simplicity, we denote $J=$ Jac $(f_1,f_2)$.\\

So we write 
\begin{eqnarray}
J &=& \langle 2\times 2 \text{ minors of Jacobian matrix}\rangle\nonumber\\
&=&\langle (b_2-b_1)x_1x_2,\cdots, (b_n-b_1)x_1x_n,\cdots, (b_n-b_{n-1})x_{n-1}x_n\rangle\nonumber\\
&=& \bigcap_{i=1}^n\langle x_1,x_2,\cdots,x_{i-1},x_{i+1},\cdots,x_n\rangle\qquad\cdots (5.1)\nonumber
\end{eqnarray}

Our goal is now to compute the codimension of $J$ and relate this to minrank$(f_1,f_2)$.  We will compute the codimension of $J$ by first producing an explicit primary decomposition of $J$.\\

Hence, to complete the proof of this theorem, it suffices to prove two auxiliary results.  First, we will show that codim $J$ equals minrank$(f_1,f_2)$, in Lemma \ref{minrank}.  Second, in Lemma \ref{coll_strength}, we will show that minrank$(f_1,f_2) \geq 4$ if the collective strength of $f_1, f_2, f_3$ is at least 2.  Then, by applying Lemma \ref{prime}, this will complete the proof.

\end{proof}

%Now, consider the example $n=8$ and $B=(0,0,0,1,1,2,3,4)$. Macaulay 2 gives the following result:\\
%
%$I_{i_1}(x_7 , x_8 , x_4 , x_5 , x_1 , x_2 , x_3 )\cap I_{i_2}(x_6 , x_8, x_4 , x_5 , x_1 , x_2 , x_3 )\cap I_{i_3}(x_6 , x_7 , x_4 , x_5 , x_1 , x_2 , x_3 )$\\
%$\cap I_{i_4}(x_6 , x_7, x _8, x_1, x_2 , x_3 ) \cap I_{i_5}(x_6 , x_7 , x_8 , x_4 , x_5 ).$\\
% 
%\begin{center}
%\begin{tabular}{c|c}
%ideal & variables removed\\
%\hline
%$I_{i_1}$ & $x_6$ \\
%$I_{i_2}$ & $x_7$ \\
%$I_{i_3}$ & $x_8$ \\
%$I_{i_4}$ & $x_4, x_5$\\
%$I_{i_5}$ & $x_1, x_2,x_3$\\
%\end{tabular}
%\end{center}
%
%So
%$$\text{ codim }J=\text{ codim } I_{i_5}=5$$\\

\begin{lemma}
\label{minrank} With notation as in Theorem \ref{n32}, we claim that codim $J$ = minrank$(f_1,f_2)$.
\end{lemma}

\begin{proof}
The ideal $J$ depends on the number of pairs $b_i, b_j$ where $b_i = b_j$.  To streamline the computations about $J$, it will therefore be convenient to change notation in a way that focus on which coefficients of $f_2$ are the same.  So we suppose that coefficients of $f_2$ are $\{\alpha_1,\cdots, \alpha_r\}$ where the $\alpha_i$ are all distinct and $r\leq n$.  We let $\lambda_i$ denote the number of times that $\alpha_i$ appears as a coefficient of $f_2$. And we let $\mu_i = \sum_{j=1}^i \lambda_j$. After permuting the variables, we can then rewrite

%Let $$\lambda=(\lambda_1,\cdots,\lambda_r)$$ be the partition where $\lambda_1+\cdots+\lambda_r=n$.\\

$$f_2=\alpha_1(\sum_{j=1}^{\mu_1}x_j^2)+\alpha_2(\sum_{j=\mu_1+1}^{\mu_2}x_j^2)+\cdots+\alpha_r(\sum_{j=\mu_{r-1}+1}^{\mu_r}x_j^2).$$\\

It will be further convenient to relabel our variables in accordance with these $\lambda_i$.  So for each $i = 1,\cdots, r$, we introduce variables $z_{i, j}$ for $1 \leq j \leq \lambda_i$.\\

With this convention, we can now express the Jacobian matrix as:\\

\[\begin{bmatrix}
z_{1,1} & \cdots & z_{1,\lambda_1} & \vline & z_{2,1} & \cdots & z_{2,\lambda_2} & \vline & \cdots & \vline & z_{k,1} & \cdots  & z_{k,\lambda_k}\\
\alpha_1z_{1,1} & \cdots & \alpha_1z_{1,\lambda_1} & \vline & \alpha_2z_{2,1} & \cdots & \alpha_2z_{2,\lambda_2} & \vline & \cdots & \vline & \alpha_kz_{k,1} & \cdots & \alpha_kz_{k,\lambda_k}\\
\end{bmatrix}
\]\\

%Let $I=\langle x_{11},\cdots,x_{1n_1}, x_{21},\cdots,x_{2n_2},\cdots, x_{k1}\cdots x_{kn_k}\rangle$. The primary decomposition is $$I\backslash (\{x_{1i}x_{1j}\}_{i\neq j}\cup\{x_{2i}x_{2j}\}_{i\neq j}\cup\cdots\cup\{x_{ki}x_{kj}\}_{i\neq j}) $$ $$=(I\backslash \{x_{1i}x_{1j}\}_{i\neq j}) \cap (I\backslash \{x_{2i}x_{2j}\}_{i\neq j})\cap\cdots\cap (I\backslash \{x_{ki}x_{kj}\}_{i\neq j}) $$
%
%i.e. the generators of $2\times 2$ minors must come from the columns of different blocks. i.e. $i\neq j$.\\
%
 
By (5.1), we now see that the Jacobian ideal $J$ can be expressed as:
%$$x_{im}\alpha_jx_{jl}-\alpha_i x_{im}x_{jl}=(\alpha_j-\alpha_i)x_{im}x_{jl},$$ Jacobian ideal can be expressed as 
$$J=\langle z_{i,m}z_{j,l} \mid i\neq j, 1\leq m\leq\lambda_i, 1\leq l\leq \lambda_j\rangle$$

Now we are ready to compute a primary decomposition of J.  We let $I_t$ be the prime ideal $I_t= \langle z_{a ,b} | a\neq t \rangle$.  And we let $$I = \bigcap_{t=1}^r I_t$$\\

We claim that $I=J$, first we prove $J\subseteq I$. Without loss of generality, prove that $z_{i,1}z_{j,1}\in I_t$ for all $t$.\\

Case 1: $t=i$\\

When $t=i$, we have $I_t=\langle z_{a,b}\mid a\neq i\rangle$ which means $z_{j,1}\in I_t$. Hence the product $z_{i,1}z_{j,1}$ is in $I_t$ for all $n$.\\

Case 2: $t=j$\\

When $t=j$, we have $I_t=\langle z_{a,b}\mid a\neq j\rangle$ which means $z_{i,1}\in I_t$. Hence the product $z_{i,1}z_{j,1}$ is in $I_t$ for all $t$.\\

Case 3: $t\neq i,j$\\

This is clear because both $z_{i,1}$ and $z_{j,1}$ are in $I_t$, so is their product.\\

Thus, we have proved that $J\subseteq I$.\\

Now we prove $I\subseteq J$. Indeed, if a monomial $m$ does not lie $J$, then $m$ must not divisible by any pair of variables $z_{i,a}$ and $z_{j,b}$ with $i\neq j$.  In other words, for some $i$, m must be a monomial only involving the variables $z_{i,1}, \cdots, z_{i,\lambda_i}$.  This means that $m \in \langle z_{i,1},\cdots, z_{i,\lambda_i}\rangle$. Then we have $m\not\in I_t=\langle z_{a,b}\mid a\neq t\rangle$, which implies that $m\not\in I$ as desired. Therefore, we have that $I=J$.\\

Hence, we have codim $J$ = codim $I=\min_t\{\text{codim }I_t\}$.  Finally, we need to compare minrank$(f_1,f_2)$ to codim $I$.  Given our special form for $f_1$ and $f_2$, the smallest rank of a quadric will be obtained by a $k$-linear combination of $f_1$ and $f_2$ which has the most coefficients equal to zero.  In other words, minrank$(f_1,f_2)$ is $n - \max {\lambda_t , 1\leq t \leq r}$.  But we note that $n - \lambda_t$ is precisely codim$(I_t)$, and we conclude that codim$(J) =$ minrank$(f_1,f_2)$ as desired.
\end{proof}

\begin{lemma}
\label{coll_strength} If collective strength of $f_1,f_2,f_3$ is at least 2 then minrank$(f_1,f_2)$ is at least 4.
\end{lemma}

\begin{proof}
This follows from Remark \ref{rank_strength}.
\end{proof}

\begin{prop} \label{nn32}There exist homogeneous polynomials $f_1,f_2,f_3$ with collective strength 1 that do not form a regular sequence. This implies $N(3,2)\geq 2$.
\end{prop}

\begin{proof}
Consider the $3\times 2$ matrix
\[\begin{bmatrix}
x_{11} & x_{12}\\
x_{21} & x_{22}\\
x_{31} & x_{32}\\
\end{bmatrix}
\]\\
The minors of the matrix $$f_1=x_{21}x_{32}-x_{31}x_{22},\quad f_2=x_{11}x_{32}-x_{31}x_{12},\quad f_3=x_{11}x_{22}-x_{21}x_{12}.$$ do not form a regular sequence by the discussion in Section 4. So to complete the proof, we need to show that $f_1, f_2, f_3$ have collective strength 1.\\

Let $a(x_{21}x_{32}-x_{31}x_{22})+b(x_{11}x_{32}-x_{31}x_{12})+c(x_{11}x_{22}-x_{21}x_{12})$ be the linear combination of $f_1,f_2,f_3$ for some $a,b,c$ in a field $k$. If $(a,b,c)=(0,0,0)$, there's nothing to prove. So we can assume that at least one of $a,b,c$ is nonzero. By symmetry, without loss of generality we can assume that $a\neq 0$.\\

By definition, the collective strength of $f_1,f_2,f_3$ is the minimal strength of a non-trivial linear combination of $f_1,f_2,f_3$, so is obviously $\leq 1$. Want to prove that it is precisely 1, hence $N(3,2)\geq 2$.\\

The matrix of this quadratic polynomial is
 
\[\begin{bmatrix}
0 & 0 & 0 & 0 & a & b\\
0 & 0 & 0 & -a & 0 & c\\
0 & 0 & 0 & -b & -c & 0\\
0 &-a & -b & 0 & 0 & 0\\
a & 0 & -c & 0 & 0 & 0\\
b & c & 0 & 0 & 0 & 0\\
\end{bmatrix}.
\]\\
Since we have $a\neq 0$. Reducing to echelon form, we get
\[\begin{bmatrix}
1 & 0 & -\frac{c}{a} & 0 & 0 & 0\\
0 & 1 & \frac{b}{a} & 0 & 0 & 0\\
0 & 0 & 0 & 1 & 0 & -\frac{c}{a}\\
0 & 0 & 0 & 0 & 1 & \frac{b}{a}\\
0 & 0 & 0 & 0 & 0 & 0\\
0 & 0 & 0 & 0 & 0 & 0\\
\end{bmatrix}.
\]\\

So the rank is $\geq 4$, i.e. the rank of the linear combination of $f_1,f_2,f_3$ has rank $\geq 4$, which means its strength is $\geq 1$ by Remark \ref{rank_strength}. Therefore, it is exactly 1. Hence, $N(3,2)\geq 2$.\\
\end{proof}
%In fact, since collective strength is $\leq 1$, need to show that $\neq 1$, hence $=1$.\\

%Alternatively, since the collective strength is $\leq 1$, let $q=q_1q_2+q_3q_4$, the biggest possible rank is 4: If $q_i$ are all linearly independent, then $q=x_1^2+x_2^2+x_3^2+x_4^2=x_1x_2+x_3x_4$ (by changing coordinate); if not linearly independent, then rank $<4$, say rank $=3$, then $q=x_1^2+x_2^2+x_3^2=x_1x_2+x_3^2$.\\
%
%Or even easier, by definition of collective strength (the minimal strength of the non-trivial linear combination of $\{f_i\}$). So let $b,c=0$ for example, the collective strength is 1. This proves that $N(3,2)\geq 2$.

Finally, combining Theorem \ref{n32} and Proposition \ref{nn32} proves that $N(3,2)=2$.\\

% SECTION 6

\section{Homogeneous Polynomials and Column Grading}

We introduce column grading in this section. And in Lemma \ref{homog} we prove that a cubic homogeneous polynomial of strength 1 can be written as a linear combination of homogeneous polynomials of the same strength in column grading.\\

\begin{defn}
We imagine the variables of the $\C\C[x_{11},\cdots , x_{43}]$ as corresponding to entries in a generic $4\times 3$ matrix:

\[M=\begin{bmatrix}
x_{11} & x_{12} & x_{13}\\
x_{21} & x_{22} & x_{23}\\
x_{31} & x_{32} & x_{33}\\
x_{41} & x_{42} & x_{43}\\
\end{bmatrix}
\]

We give this ring a $\Z^3$ grading by its columns.  Namely, $deg(x_{i1})=(1,0,0)$ for $1\leq i \leq 4$ and similarly $deg(x_{i2})=(0,1,0)$ and $deg(x_{i3})=(0,0,1)$.
\end{defn}

%\textbf{Nonstandard grading}\\
%
%Define a grading $deg: \C[x_{11},\cdots,x_{33}]\rightarrow \Z^3\times\Z^3$ by
%$$(x_{ij})\mapsto
%\begin{pmatrix}
%\vdots\\
%1\\
%\vdots\\
%\end{pmatrix}
%i^{th}\oplus
%\begin{pmatrix}
%\vdots \\
%1\\
%\vdots\\
%\end{pmatrix}
%j^{th}
%$$

%Consider
%$$(x_{23})\mapsto
%\begin{pmatrix}
%0\\
%1\\
%0\\
%\end{pmatrix}
%\oplus
%\begin{pmatrix}
%0 \\
%0\\
%1\\
%\end{pmatrix},
%\qquad
%(x_{12})\mapsto
%\begin{pmatrix}
%1\\
%0\\
%0\\
%\end{pmatrix}
%\oplus
%\begin{pmatrix}
%0\\
%1\\
%0\\
%\end{pmatrix}
%$$

\begin{eg}
$(x_{11}+x_{23})$ is not homogeneous because $x_{11}$ and $x_{23}$ have different degree. They are $\deg(x_{11})=(1,0,0)$ for $1\leq i \leq 4$ and $\deg(x_{23})=(0,0,1)$.
\end{eg}

\begin{eg}
If we choose degree $(1,1,0)$ then $S_{(1,1,0)}$ has a basis consisting of products $x_{i1}x_{j2}$ for any $1\leq i \leq 4$ and any $1\leq j \leq 4$.
\end{eg}

In the following section, we will be interested in studying the strength of $3\times 3$ minors of the above matrix.  We note that any such determinant has degree $(1,1,1)$ in this grading.\\

%\begin{eg} Define $$S_{
%\begin{pmatrix}
%1\\
%1\\
%0\\
%\end{pmatrix}
%\oplus
%\begin{pmatrix}
%1\\
%0\\
%1\\
%\end{pmatrix}
%}
%=\text{ vector space of homogeneous degree} \begin{pmatrix}
%1\\
%1\\
%0\\
%\end{pmatrix}
%\oplus
%\begin{pmatrix}
%1\\
%0\\
%1\\
%\end{pmatrix} \text{polynomials.}$$
%
%Then the basis is $\{x_{11}x_{23},x_{13}x_{21}\}$ because $\deg(x_{11}x_{23})=\deg(x_{11})+\deg(x_{23})$ which is
%$$\begin{pmatrix}
%1\\
%0\\
%0\\
%\end{pmatrix}
%\oplus
%\begin{pmatrix}
%1\\
%0\\
%0\\
%\end{pmatrix}
%+
%\begin{pmatrix}
%0\\
%1\\
%0\\
%\end{pmatrix}
%\oplus
%\begin{pmatrix}
%0\\
%0\\
%1\\
%\end{pmatrix}
%=
%\begin{pmatrix}
%1\\
%1\\
%0\\
%\end{pmatrix}
%\oplus
%\begin{pmatrix}
%1\\
%0\\
%1\\
%\end{pmatrix}
%$$
%
%Similarly, $\deg(x_{13}x_{21})$ has degree
%$$\begin{pmatrix}
%1\\
%0\\
%0\\
%\end{pmatrix}
%\oplus
%\begin{pmatrix}
%0\\
%0\\
%1\\
%\end{pmatrix}
%+
%\begin{pmatrix}
%0\\
%1\\
%0\\
%\end{pmatrix}
%\oplus
%\begin{pmatrix}
%1\\
%0\\
%0\\
%\end{pmatrix}
%=
%\begin{pmatrix}
%1\\
%1\\
%0\\
%\end{pmatrix}
%\oplus
%\begin{pmatrix}
%1\\
%0\\
%1\\
%\end{pmatrix}
%$$
%\end{eg}
%Hence, given $$\Delta=\det
%\begin{pmatrix}
%x_{11} & x_{12} & x_{13}\\
%x_{21} & x_{22} & x_{23}\\
%x_{31} & x_{32} & x_{33}\\
%\end{pmatrix}
%$$
%By the symmetry of determinants, we have $$deg(\Delta)=
%\begin{pmatrix}
%1\\
%1\\
%1\\
%\end{pmatrix}
%\oplus
% \begin{pmatrix}
%1\\
%1\\
%1\\
%\end{pmatrix}
%$$

%We say that $\Delta$ is $^\prime$homogeneous$^\prime$ in this grading.\\

Next, we turn our attention to analyzing ideals generated by a pair of linear forms, up to appropriate symmetries, as this will be used in Section 7.

\begin{lemma}
Up to the natural $GL_4\times GL_3$ -action on $\C[x_{11}, … , x_{43}]$, any ideal I which is homogeneous in the $\Z^3$ grading, and which is generated by a pair of linearly independent linear forms, is equivalent to one of the following:
\begin{itemize}
\item $(x_{11},x_{21})$
\item $(x_{11},x_{12})$
\item $(x_{11},x_{22})$
\end{itemize}
\end{lemma}

\begin{proof}

%Recall that in linear algebra, two $m\times n$ matrices $A$ and $B$ are called equivalent if $$B=Q^{-1}AP$$ for some invertible $n\times n$ matrix $P$ and some invertible $m\times m$ matrix $Q$. Equivalent matrices represent the same linear transformation $V\rightarrow W$ under two different choices of a pair of bases of $V$ and $W$, with $P$ and $Q$ being the change of basis matrices in V and W respectively. Thus, for two rectangular matrices of the same size, their equivalence can also be characterized by the following conditions
%\begin{itemize}
%\item The matrices can be transformed into one another by a combination of elementary row and column operations.
%\item Two matrices are equivalent if and only if they have the same rank.
%\end{itemize}
%
%Note that this is still true for column grading since the transformation from one into another can be done by elementary row and column operations.

We start with an ideal generated by two linearly independent linear forms $\ell_1, \ell_2$, each of which is homogeneous in the column grading.  Without loss of generality, we can assume that $ \ell_1 = x_{11}$.  If $\deg(\ell_2) = \deg(\ell_1)$ then we can assume that $\ell_2 = x_{21}$, which is the first case listed in the lemma.  If $deg(\ell_2) \ne \deg(\ell_1)$ then we can assume that $\ell_2 = x_{i2}$ for some i.  Again there are two cases: if $i=1$ then we are in one the second case listed in the lemma.  If $i\ne 1$, then we can assume that $i =2$ and our forms are $x_{11}$ and $x_{22}$ and so we are in the third case of the lemma.\\

%For example, in column grading $(x_{1,1},x_{2,2})\sim (x_{1,1},x_{3,2})$ are of course equivalent since they are in the same column. And $(x_{1,1},x_{2,2})\sim (x_{1,1},x_{3,3})$ are equivalent via column operation:
%
%$$
%\begin{bmatrix}
%\bullet & 0 & 0\\
%0 & \bullet & 0 \\
%0 & 0 & 0\\
%0 & 0 & 0\\
%\end{bmatrix}
%\sim 
%\begin{bmatrix}
%\bullet & 0 &0\\
%0 & 0 & 0\\
%0 & 0 & \bullet \\
%0 & 0 & 0\\
%\end{bmatrix}
%$$
%Thus, we obtain an equivalence class: $$\{ (x_{1,1},x_{i,j}): \text{for } 2\leq i\leq 4\text{ and }2\leq j\leq 3\}$$
%To find other equivalence class, note that $(x_{1,1},x_{1,2})\not\sim (x_{1,1},x_{2,2})$ since they cannot be transformed into one another by either row or column operation:
%$$
%\begin{bmatrix}
%\bullet & \bullet & 0\\
%0 & 0 & 0 \\
%0 & 0 & 0\\
%0 & 0 & 0\\
%\end{bmatrix}
%\not\sim 
%\begin{bmatrix}
%\bullet & 0 &0\\
%0 & \bullet & 0\\
%0 & 0 & 0\\
%0 & 0 & 0\\
%\end{bmatrix}
%$$
%
%But we have $(x_{1,1},x_{1,2})\sim (x_{1,1},x_{1,3})$ via column operation: 
%$$
%\begin{bmatrix}
%\bullet & \bullet & 0\\
%0 & 0 & 0 \\
%0 & 0 & 0\\
%0 & 0 & 0\\
%\end{bmatrix}
%\sim 
%\begin{bmatrix}
%\bullet & 0 &\bullet\\
%0 & 0 & 0\\
%0 & 0 & 0\\
%0 & 0 & 0\\
%\end{bmatrix}
%$$
%
%So we have an equivalence class: $$\{(x_{1,1},x_{1,i}): \text{for }i=2,3\}$$ Similarly, we have another equivalence class: $$\{(x_{1,1},x_{j,1}):\text{for }j=2,3\}$$
%

Hence, when fixing $x_{11}$ there are three equivalent classes with the following representatives: $$\{(x_{11},x_{21}), (x_{11}, x_{22}), (x_{11},x_{12})\}.$$

%Since the grading map $deg: \C[x_{11},\cdots,x_{33}]\rightarrow \Z^3\times\Z^3$ sends homogeneous element in $\C[x_{11},\cdots,x_{33}]$ to homogeneous element in $\Z^3\times \Z^3$. Hence there is only one equivalence class in $\C[x_{11},\cdots,x_{33}]$ which is $^\prime$homogeneous$^\prime$ polynomials in a grading sense.\\
\end{proof}

This lemma will be used in proving $N(3,3)>2$ where we apply column grading.\\

% SECTION 7

\section{The Bound on Collective Strength when $r=3$ and $d=3$}
%Let $\Delta$ be a form of degree $\begin{pmatrix}
%1\\
%1\\
%1\\
%\end{pmatrix}$ in column grading of strength one, i.e. $\Delta=ab+cd$, where $a,c$ are linear and $b,d$ are quadratic and they are all homogeneous in standard grading by Definition \ref{strength}. In Lemma \ref{homog} we prove that there exist $a^\prime,b^\prime,c^\prime,d^\prime$ homogeneous in column grading such that $\Delta=a^\prime b^\prime +c^\prime d^\prime$.

In this section, we prove that there exist cubic homogeneous polynomials $f_1,f_2,f_3$ with collective strength 2 that is not a regular sequence. i.e. $N(3,3)>2$.\\

As mentioned in Section 4, the $n\times n$ minors of the $(n+1) \times n$ matrix are not a regular sequence. So let

\[M=\begin{bmatrix}
x_{11} & x_{12} & x_{13}\\
x_{21} & x_{22} & x_{23}\\
x_{31} & x_{32} & x_{33}\\
x_{41} & x_{42} & x_{43}\\
\end{bmatrix}
\] be a $4\times 3$ matrix.\\

And let 
\[\Delta_1=\begin{bmatrix}
x_{21} & x_{22} & x_{23}\\
x_{31} & x_{32} & x_{33}\\
x_{41} & x_{42} & x_{43}\\
\end{bmatrix},\quad
\Delta_2=\begin{bmatrix}
x_{11} & x_{12} & x_{13}\\
x_{31} & x_{32} & x_{33}\\
x_{41} & x_{42} & x_{43}\\
\end{bmatrix}
\]\\
\[\Delta_3=\begin{bmatrix}
x_{11} & x_{12} & x_{13}\\
x_{21} & x_{22} & x_{23}\\
x_{41} & x_{42} & x_{43}\\
\end{bmatrix},\quad
\Delta_4=\begin{bmatrix}
x_{11} & x_{12} & x_{13}\\
x_{21} & x_{22} & x_{23}\\
x_{31} & x_{32} & x_{33}\\
\end{bmatrix}
\]
be $3\times 3$ submatrices of $M$ obtained by deleting $i^{th}$ row. Then $f_1=\text{det}\Delta_1, f_2=\text{det}\Delta_2, f_3=\text{det}\Delta_3$ are $3\times 3$ minors and $f_1,f_2,f_3$ are homogeneous polynomials of degree 3 that is not a regular sequence.\\

By using Laplace expansion, we have:
\[
\Delta_1 = x_{21}(x_{32}x_{43} - x_{33}x_{42}) - x_{22}(x_{31}x_{43} - x_{33}x_{41}) + x_{23}(x_{31}x_{42} - x_{32}x_{41}).
\]
It follows that the strength of $\Delta_1$ is at most $2$ and thus the collective strength of $f_1,f_2$, and $f_3$ is also at most $2$. We want to prove that the collective strength is not 1, hence it is 2. Suppose it is 1, then $\Delta$ can be written as $\Delta=ab+cd$, where $a,c$ are linear and $b,d$ are quadrics. Note that each of $f_1,f_2,f_3$ is a cubic homogeneous polynomial of degree $\begin{pmatrix}
1\\
1\\
1\\
\end{pmatrix}$ in column grading, so is their linear combination $\Delta$.\\

Since $a$ and $c$ are linear, they are linear combination of degree $\begin{pmatrix}
1\\
0\\
0\\
\end{pmatrix}$, $\begin{pmatrix}
0\\
1\\
0\\
\end{pmatrix}$ and $\begin{pmatrix}
0\\
0\\
1\\
\end{pmatrix}$ in column grading. So $a$ can be written as $a=a_{100}+a_{010}+a_{001}$ where $a_{100}$ denotes the combination of terms from column 1,  $a_{010}$ denotes the combination of terms from column 2, $a_{001}$ denotes the combination of terms from column 3. Similarly, $c=c_{100}+c_{010}+c_{001}$.\\

Likewise, since $b,d$ are quadric, they are linear combination of degree $\begin{pmatrix}
2\\
0\\
0\\
\end{pmatrix}$, $\begin{pmatrix}
0\\
2\\
0\\
\end{pmatrix}$, $\begin{pmatrix}
0\\
0\\
2\\
\end{pmatrix}$, $\begin{pmatrix}
1\\
1\\
0\\
\end{pmatrix}$, and $\begin{pmatrix}
0\\
1\\
1\\
\end{pmatrix}$, $\begin{pmatrix}
1\\
0\\
1\\
\end{pmatrix}$ in column grading, denoted by $b_{200},b_{020},b_{002},b_{110},b_{101}, b_{011}$. Hence, $b$ can be written as $b=b_{200}+b_{020}+b_{002}+b_{110}+b_{101}+b_{011}$. Similarly, $d=d_{200}+d_{020}+d_{002}+d_{110}+d_{101}+d_{011}$.\\

Furthermore, since $\Delta=ab+cd$ can be seen as an ideal generated by two linear elements $a, c$, where $a=a_{100}+a_{010}+a_{001}, c=c_{100}+c_{010}+c_{001}$ i.e. $\Delta=\langle a,c\rangle$. By multiplying a suitable invertible $2\times 2$ matrix, $\Delta$ can be re-written as $\Delta=\langle a,a+c\rangle$ for example. Hence, without loss of generality, we can assume that $a_{100}\neq c_{100}$ and both nonzero. In fact, we can apply to other degrees, such as $a_{010}, c_{010}$ and $a_{001}, c_{001}$. That means, fixing $a_{100}$ and $c_{100}$ with $a_{100}\neq c_{100}$ and both nonzero, through a suitable linear transformations we can assume that the other two pairs $(a_{010}, c_{010})$, $(a_{001},c_{001})$ in $a$ and $c$ are either both zero or both nonzero and non-equal.\\

\begin{eg} If $a=a_{100}+a_{010}$, $c=c_{100}$, where $a_{100}\neq c_{100}$. Then we can make $(a,c)$ to be nonzero on degrees $\begin{pmatrix}
1\\
0\\
0\\
\end{pmatrix}$ and $\begin{pmatrix}
0\\
1\\
0\\
\end{pmatrix}$ for both $a$ and $c$ by multiplying an invertible matrix
\[
\begin{pmatrix}
1 & 0\\
q & 1\\
\end{pmatrix}
\]
where a general choice of $q$ will have the desired properties as below:
i.e. $$
\begin{pmatrix}
1 & 0\\
q & 1\\
\end{pmatrix}
\cdot
\begin{pmatrix}
a_{100}+a_{010}\\
c_{100}\\
\end{pmatrix}
=
\begin{pmatrix}
a_{100}+a_{010}\\
q(a_{100}+a_{010})+c_{100}\\
\end{pmatrix}
$$

%Note that the result matrix is zero in degree $\begin{pmatrix}
%0\\
%0\\
%1\\
%\end{pmatrix}$since both $a$ and $c$ are.
\end{eg}

\begin{eg} If $a=a_{100}+a_{001}$, $c=c_{100}+c_{010}$, where $a_{100}\neq c_{100}$. Then we can make $(a,c)$  nonzero in all degrees by multiplying by a general $2\times 2$ matrix to obtain the desired effect.
%
%
%\[
%\begin{pmatrix}
%1 & 1\\
%-1 & -1\\
%\end{pmatrix}
%\]
%i.e. $$
%\begin{pmatrix}
%1 & 1\\
%-1 & -1\\
%\end{pmatrix}
%\cdot
%\begin{pmatrix}
%a_{100}+a_{001}\\
%c_{100}+c_{010}
%\end{pmatrix}
%=\begin{pmatrix}
%a_{100}+c_{100}+c_{010}+a_{001}\\
%-a_{100}-c_{100}-c_{010}-a_{001}\\
%\end{pmatrix}
%$$

Thus, if at least one of $a$ and $c$ is nonzero in some certain degree, say $\begin{pmatrix}
1\\
0\\
0\\
\end{pmatrix}$, we can always make both $a$ and $c$ nonzero for that degree via some matrix transformation. Obviously, this won't work if both $a$ and $c$ are zero in that degree. Therefore, for each pair of column degree: $(a_{100}, c_{100})$, $(a_{010}, c_{010})$, $(a_{001},c_{001})$ we can assume that they are either both zero, or both nonzero and non-equal.
\end{eg}

\begin{lemma}\label{homog} Let $\Delta$ be a form of degree $\begin{pmatrix}
1\\
1\\
1\\
\end{pmatrix}$ in column grading. Assume that $\Delta$ has strength one, that is, we can write $\Delta=ab+cd$, where $a,c$ are linear (in the standard grading) and $b,d$ are quadric (in the standard grading). Then we can find $a^\prime,b^\prime,c^\prime,d^\prime$ all homogeneous in column grading, such that $\Delta=a^\prime b^\prime +c^\prime d^\prime$.
\end{lemma}

\begin{proof}
Since $a,c$ are linear and $b,d$ quadric, by discussion above we can write $$a=a_{100}+a_{010}+a_{001}$$ $$c=c_{100}+c_{010}+c_{001}$$ $$b=b_{200}+b_{020}+b_{002}+b_{110}+b_{101}+b_{011}$$ $$d=d_{200}+d_{020}+d_{002}+d_{110}+d_{101}+d_{011}$$ Since $\Delta$ has degree $(1,1,1)$ in the column grading, we have that:
\begin{eqnarray}
\Delta &=& ab+cd\nonumber\\
&=&(a_{100}+a_{010}+a_{001})(b_{200}+b_{020}+b_{002}+b_{110}+b_{101}+b_{011})\nonumber\\
&+& (c_{100}+c_{010}+c_{001})(d_{200}+d_{020}+d_{002}+d_{110}+d_{101}+d_{011})\nonumber\\
&=& a_{100}b_{011}+a_{010}b_{101}+a_{001}b_{110}+c_{100}d_{011}+c_{010}d_{101}+c_{001}d_{110}\nonumber
\end{eqnarray}

with the following equations
$$a_{100}b_{200}+c_{100}d_{200}=0\quad\cdots (1)$$
$$a_{100}b_{020}+a_{010}b_{110}+c_{100}d_{020}+c_{010}d_{110}=0\quad\cdots (2)$$
$$a_{100}b_{002}+a_{001}b_{101}+c_{100}d_{002}+c_{001}d_{101}=0\quad\cdots (3)$$
$$a_{010}b_{200}+a_{100}b_{110}+c_{010}d_{200}+c_{100}d_{110}=0\quad\cdots (4)$$
$$a_{010}b_{020}+c_{010}d_{020}=0\quad\cdots (5)$$
$$a_{010}b_{002}+a_{001}b_{011}+c_{010}d_{002}+c_{001}d_{011}=0\quad\cdots (6)$$
$$a_{001}b_{200}+a_{100}b_{101}+c_{001}d_{200}+c_{100}d_{101}=0\quad\cdots (7)$$
$$a_{001}b_{002}+c_{001}d_{002}=0\quad\cdots (8)$$
$$a_{001}b_{020}+a_{010}b_{011}+c_{001}d_{020}+c_{010}d_{011}=0\quad\cdots (9)$$\

We separate the proof into two major cases: Case 1 is where $a_{100}$, $c_{100}$ are linearly independent and Case 2 is where $a_{100}$, $c_{100}$ are linearly dependent.\\

\textbf{Case 1}: Assume that $a_{100}$, $c_{100}$ are linearly independent. Then WLOG $a_{100}= x_{11}$ and $c_{100} = x_{21}$. We will prove the lemma in this case by first proving a series of subclaims.\\

\textit{Claim 1.1}: There exists $B$ such that $b_{200} = x_{21}B$ and $d_{200} = -x_{11}B$.\\

Proof of Claim 1.1\\

Let $a_{100}=x_{11}, c_{100}=x_{21}$, then (1) becomes $$x_{11}b_{200}+x_{21}d_{200}=0$$ If $b_{200}=0$, then $d_{200}=0$ or vice versa, then there's nothing to prove and we move to Case 2. So we assume that $b_{200}\neq 0$ and $d_{200}\neq 0$. Thus, from (1) $b_{200}$ must have a factor $x_{21}$, i.e. $b_{200}=x_{21}B$ for some linear term $B$ from the first column. Similarly, $d_{200}=x_{11}D$ for some $D$ from the first column. So we can write
\begin{eqnarray}
0&=& x_{11}x_{21}B+x_{21}x_{11}D=0\quad\cdots (1)\nonumber\\
&=& x_{11}x_{21}(B+D)\nonumber
\end{eqnarray}
which implies that $B+D=0$, or $B=-D$. Hence, we proved that there exists $B$ such that $b_{200} = x_{21}B$ and $d_{200} = -x_{11}B$.\\

\textit{Claim 1.2} This further implies $b_{110}=c_{010}B$ and $d_{110}=-a_{010}B$ for some homogeneous $c_{010}$ and $a_{010}$ of degree $(0,1,0)$.\\

Proof of Claim 1.2\\

Substitute $b_{200}=x_{21}B$ and $d_{200}=x_{11}D$ into (4)\\

We have 
\begin{eqnarray}
0&=& a_{010}b_{200}+a_{100}b_{110}+c_{010}d_{200}+c_{100}d_{110}\nonumber\\
&=& a_{010}x_{21}B+x_{11}b_{110}+c_{010}x_{11}D+x_{21}d_{110}\nonumber\\
&=& x_{11}(b_{110}+c_{010}D)+x_{21}(d_{110}+a_{010}B)\nonumber
\end{eqnarray}

If two terms are canceled with each other, $b_{110}+c_{010}D$ must have the factor $x_{21}$ and $a_{010}B+d_{110}$ has factor $x_{11}$. This implies that $D=x_{21}$ and $B=x_{11}$ which is a contradiction since $B+D\neq 0$. So we must have $$b_{110}=c_{010}B,\quad d_{110}=-a_{010}B$$

\textit{Claim 1.3} This further implies $b_{020}=d_{020}=0$.\\

Proof of Claim 1.3\\

Plugging the result of Claim 1.2 into (2) we have 
\begin{eqnarray}
0&=& a_{100}b_{020}+a_{010}b_{110}+c_{100}d_{020}+c_{010}d_{110}\nonumber\\
&=& x_{11}b_{020}-a_{010}c_{010}D+x_{21}d_{020}-c_{010}a_{010}B\nonumber\\
&=& x_{11}b_{020}+x_{21}d_{020}-a_{010}c_{010}(D+B)\nonumber\\
&=& x_{11}b_{020}+x_{21}d_{020}\nonumber
\end{eqnarray}

So we must have $b_{020}=d_{020}=0$.\\

$a_{010}=c_{010}=0$\\

This means that $d_{020}$ must be a multiple of $x_{11}$ but that is impossible because of the column grading. Thus $d_{020}=0$ and similarly $b_{020}= 0$.\\

\textit{Claim 1.4} This further implies that $b_{101}=c_{001}B$ and $d_{101}=-a_{001}B$ for some homogeneous $c_{001}$ and $a_{001}$ of degree $(0,0,1)$ and that $b_{002}=c_{002}=0$.\\

Proof of Claim 1.4\\

The same arguments as in the proofs of Claim 1.2 and 1.3 apply here.\\

Thus, from our initial assumption that $a_{100}$, $c_{100}$ were linearly independent, Claims 1.2 and 1.4 allow us to reduce to the case where $a=a_{100}=x_{11}$ and $c=c_{100}=x_{21}$. Making this substitution into the expression for $\Delta$ in terms of the $a_{ijk}$ and so on, we then have $$\Delta= a_{100}b_{011}+c_{100}d_{011}$$ and this yields the desired expression for $\Delta$ which is homogeneous in the column grading.\\

We have thus completely handled Case 1, which is where $a_{100}$ and $c_{100}$ are linearly independent. By symmetry we have also handled the cases where $a_{010}$ and $c_{010}$ are linearly independent or when $a_{001}$ and $c_{001}$. So we have reduced to the case where each of these pairs is linearly dependent.\\

\textbf{Case 2} We have that $a_{100}$, $c_{100}$ are linearly dependent, and similarly for $a_{010}$, $c_{010}$ and $a_{001}$, $c_{001}$. By choosing a different basis for the vector space spanned by $a$ and $c$ (and altering $b$ and $d$ accordingly), we can assume that there exist scalars $\lambda_1$, $\lambda_ 2$, $\lambda_3$ such that $$c_{100}=\lambda_1 a_{100}\text{ and }c_{010}=\lambda_2 a_{010}\text{ and } c_{001}=\lambda_3 a_{001}.$$

Moreover, by replacing $c$ by $c-\lambda_1a$ (and again altering $b$ and $d$ accordingly), we can also assume that $\lambda_1=0$ and thus $c_{100}=0$.\\

Substituting these expressions into our equations for $\Delta$ and equation (2) from above, we get $$\Delta = a_{100}b_{011}+a_{010}(b_{101}+\lambda_2 d_{101})+a_{001}(b_{11}+\lambda_3 d_{110})$$
$$0=a_{100}b_{020}+a_{010}(b_{110} +\lambda_2d_{110})$$ If any of $a_{100}$, $a_{010}$ or $(b_{110}+\lambda_2d_{110})$ is zero, then substituting this into the expression for $\Delta$ gives an expression $\Delta= a'b' + c'd'$ of the desired form. If these are all nonzero, then we can write $$(b_{110} + \lambda_2d_{110}) = a_{100}f$$ for some $f$ of degree $(0,1,0)$. Substituting this into the expression for $\Delta$, we have
\begin{eqnarray}
\Delta &=& a_{100}b_{011}+a_{010}(b_{101}+\lambda_2d_{101})
+a_{001}(b_{110}+\lambda_3d_{110})\nonumber\\
&=& a_{100}b_{011}+a_{010}(b_{101}+\lambda_2d_{101})+a_{001}a_{100}f\nonumber\\
&=& a_{100}(b_{011}+a_{001}f)+a_{010}(b_{101}+\lambda_2d_{101})\nonumber
\end{eqnarray}
which is an expression $\Delta= a'b'+c'd'$ of the desired form.

\end{proof}

\begin{thm} There exist cubic homogeneous polynomials $f_1,f_2,f_3$ with collective strength 2 that do not form a regular sequence, i.e. $N(3,3)>2$.
\end{thm}

\begin{proof}
Let $f_i=\det\Delta_i$ for $i=1,2,3$. By the arguments in Section 4, we know that $f_1, f_2, f_3$ do not form a regular sequence, and that they have collective strength at most 2.  So to prove this theorem, it suffices to prove that the collective strength of $f_1, f_2$, and $f_3$ is at least 2.\\

Let 

$\Delta=c_1\Delta_1+c_2\Delta_2+c_3\Delta_3=
c_1\Big (x_{21}(x_{32}x_{43}-x_{33}x_{42})-x_{22}(x_{31}x_{43}-x_{41}x_{33})+x_{23}(x_{31}x_{42}-x_{32}x_{41})\Big )+c_2\Big (x_{11}(x_{32}x_{43}-x_{33}x_{42})-x_{12}(x_{31}x_{43}-x_{33}x_{41})+x_{13}(x_{31}x_{42}-x_{32}x_{41})\Big )+c_3\Big (x_{11}(x_{22}x_{43}-x_{42}x_{23})-x_{12}(x_{21}x_{43}-x_{41}x_{23})+x_{13}(x_{21}x_{42}-x_{41}x_{22})\Big )$\\

Suppose that this collective strength is 1.  Then there is some $\mathbb C$-linear combination $\Delta = c_1f_1 + c_2f_2 + c_3f_3$ which has strength 1, i.e. $\Delta = ab+cd$, with $a,c$ linear.  By Lemma 7.3 we can assume that $a$ and $c$ are homogeneous in the column grading.  By Lemma 6.4, we can reduce to the cases where $a = x_{11}$ and $c = x_{22}$ or $x_{12}$ or $x_{21}$.  We now perform an explicit computation to show that this is impossible.\\

%Then $\Delta$ has collective strength $\leq 2$. Want to show that it is not 1, hence is 2. Suppose it is 1, $\Delta=ab+cd$ can be seen as an ideal generated by two linear elements $a, c$, where $a=a_{100}+a_{010}+a_{001}, c=c_{100}+c_{010}+c_{001}$ i.e. $\Delta=\langle a,c\rangle$. By Lemma \ref{homog}, we only need to consider the cases where the generators are homogeneous in column grading.\\
%
%Since $\Delta=a^\prime b^\prime +c^\prime d^\prime$ where $a^\prime, b^\prime, c^\prime, d^\prime$ are homogeneous, $a^\prime, c^\prime$ are linear and $b^\prime, d^\prime$ are quadric. Since the linear form $a^\prime$ has degree $\begin{pmatrix}
%1\\
%0\\
%0\\
%\end{pmatrix}$, $\begin{pmatrix}
%0\\
%1\\
%0\\
%\end{pmatrix}$ or $\begin{pmatrix}
%0\\-
%0\\
%1\\
%\end{pmatrix}$ in column grading, without loss of generality, assume that $a^\prime$ has degree $\begin{pmatrix}
%1\\
%0\\
%0\\
%\end{pmatrix}$ and $c^\prime$ has degree $\begin{pmatrix}
%1\\
%0\\
%0\\
%\end{pmatrix}$ or $\begin{pmatrix}
%0\\
%1\\
%0\\
%\end{pmatrix}$.\\

%
%By the discussion at the end of Section 6. There are two equivalence classes : $$I_1=(x_{11},x_{21}),\quad I_3=(x_{11},x_{22})$$
%We want to show that $\Delta/I_i$ is not zero.\\

Recall that\\

$\Delta=c_1\Delta_1+c_2\Delta_2+c_3\Delta_3=
c_1\Big (x_{21}(x_{32}x_{43}-x_{33}x_{42})-x_{22}(x_{31}x_{43}-x_{41}x_{33})+x_{23}(x_{31}x_{42}-x_{32}x_{41})\Big )+c_2\Big (x_{11}(x_{32}x_{43}-x_{33}x_{42})-x_{12}(x_{31}x_{43}-x_{33}x_{41})+x_{13}(x_{31}x_{42}-x_{32}x_{41})\Big )+c_3\Big (x_{11}(x_{22}x_{43}-x_{42}x_{23})-x_{12}(x_{21}x_{43}-x_{41}x_{23})+x_{13}(x_{21}x_{42}-x_{41}x_{22})\Big )$\\

\textbf{Case 1} $I_1=(x_{11},x_{12})$\\

We consider $\Delta$ modulo $I_1$.  Imposing the relations $x_{11}=x_{12}=0$, $\Delta$ reduces to\\

$c_1\Big (x_{21}(x_{32}x_{43}-x_{33}x_{42})-x_{22}(x_{31}x_{43}-x_{41}x_{33})+x_{23}(x_{31}x_{42}-x_{32}x_{41})\Big )+c_2\Big (x_{13}(x_{31}x_{42}-x_{32}x_{41})\Big )+c_3\Big (x_{13}(x_{21}x_{42}-x_{41}x_{22})\Big )=0$.\\

This can be re-written as\\

$\Delta/I_1=c_1 x_{21}x_{32}x_{43}-c_1x_{32}x_{33}x_{42}-c_1x_{22}x_{31}x_{43}+c_1x_{22}x_{41}x_{33}+c_1x_{23}x_{31}x_{42}-c_1x_{23}x_{32}x_{41}+c_2x_{13}x_{31}x_{42}-c_2x_{13}x_{32}x_{41}+c_3x_{13}x_{21}x_{42}-c_3x_{13}x_{41}x_{22}=0$.\\

Note that these 10 terms: $x_{21}x_{32}x_{43}, x_{32}x_{33}x_{42}, x_{22}x_{31}x_{43}, \cdots$ are linearly independent, so their coefficients $c_1, c_2$ and $c_3$ must be zero. This is a contradiction because $\Delta=c_1\Delta_1+c_2\Delta_2+c_3\Delta_3\neq 0$. Therefore, we conclude that $\Delta/I_1\neq 0$, which means, $\Delta\not\in I_1=(x_{11},x_{12})$.\\

\textbf{Case 2} $I_2=(x_{11},x_{21})$\\

We consider $\Delta$ modulo $I_2$.  Imposing the relations $x_{11}=x_{21}=0$, $\Delta$ reduces to\\

$c_1\Big (-x_{22}(x_{31}x_{43}-x_{41}x_{33})+x_{23}(x_{31}x_{42}-x_{32}x_{41})\Big )+c_2\Big (-x_{12}(x_{31}x_{43}-x_{33}x_{41})+x_{13}(x_{31}x_{42}-x_{32}x_{41})\Big )+c_3\Big (-x_{12}(-x_{41}x_{23})+x_{13}(-x_{41}x_{22})\Big )=0$\\

This can be re-written as\\

$-c_1x_{22}x_{31}x_{43}+c_1x_{22}x_{41}x_{33}+c_1x_{23}x_{31}x_{42}-c_1x_{23}x_{32}x_{41}-c_2x_{12}x_{31}x_{43}+c_2x_{12}x_{33}x_{41}+c_2x_{13}x_{31}x_{42}-c_2x_{13}x_{32}x_{41}+c_3x_{12}x_{41}x_{23}-c_3x_{13}x_{41}x_{22}=0$.\\

Note that these 10 terms: $x_{22}x_{31}x_{43}, x_{22}x_{41}x_{33}x, x_{23}x_{31}x_{42}, \cdots$ are linearly independent, so their coefficients $c_1, c_2$ and $c_3$ must be zero. This is a contradiction because $\Delta=c_1\Delta_1+c_2\Delta_2+c_3\Delta_3\neq 0$. Therefore, we conclude that $\Delta/I_2\neq 0$, which means, $\Delta\not\in I_2=(x_{11},x_{21})$.\\

\textbf{Case 3} $I_3=(x_{11},x_{22})$\\

We consider $\Delta$ modulo $I_3$.  Imposing the relations $x_{11}=x_{22}=0$, $\Delta$ reduces to\\

$c_1\Big (x_{21}(x_{32}x_{43}-x_{33}x_{42})+x_{23}(x_{31}x_{42}-x_{32}x_{41})\Big )+c_2\Big (-x_{12}(x_{31}x_{43}-x_{33}x_{41})+x_{13}(x_{31}x_{42}-x_{32}x_{41})\Big )+c_3\Big (-x_{12}(x_{21}x_{43}-x_{41}x_{23})+x_{13}(x_{21}x_{42})\Big )=0$\\

This can be re-written as\\

$c_1x_{21}x_{32}x_{43}-c_1x_{21}x_{33}x_{42}+c_1x_{23}x_{31}x_{42}-c_1x_{23}x_{32}x_{41}-c_2x_{12}x_{31}x_{43}+c_2x_{12}x_{33}x_{41}+c_2x_{13}x_{31}x_{42}-c_2x_{13}x_{32}x_{41}-c_3x_{12}x_{21}x_{43}+c_3x_{12}x_{41}x_{23}+c_3x_{13}x_{21}x_{42}=0$.\\

Note that these 10 terms: $x_{21}x_{32}x_{43},x_{21}x_{33}x_{42},x_{23}x_{31}x_{42}, \cdots$ are linearly independent, so their coefficients $c_1, c_2$ and $c_3$ must be zero. This is a contradiction because $\Delta=c_1\Delta_1+c_2\Delta_2+c_3\Delta_3\neq 0$. Therefore, we conclude that $\Delta/I_3\neq 0$, which means, $\Delta\not\in I_3=(x_{11},x_{22})$.\\

Thus, the collective strength of $f_1,f_2,f_3$ is not 1, hence is 2. Therefore, we proved that there exists cubic homogeneous polynomials $f_1,f_2,f_3$ with collective strength 2 that is not a regular sequence. i.e. $N(3,3)>2$.
\end{proof}

%\newpage
%\begin{center}
%\textbf{8. REFERENCES}\\
%\end{center}
%
%\begin{enumerate}[(1)]
%\item T. Ananyan and M. Hochster, \textit{Ideals generated by quadratic polynomials}, Math. Research Letters 19 (2012), pp. 233–244.
%\item T. Ananyan and M. Hochster, \textit{Small subalgebras of polynomial rings and Stillman’s conjecture}, preprint, arXiv:1610.09268.
%\item D. Erman, S. Sam and A. Snowden, \textit{Big polynomial rings and Stillman's conjecture}, arXiv:1801.09852v2 [math.AC], February 2018.
%\item T. Ananyan and M. Hochster, \textit{Strength conditions, small subalgebras, and Stillman bounds in degree}, preprint, arXiv:1810.00413 [math.AC], September 30, 2018.
%\end{enumerate}
%


\newpage
\begin{bibdiv}
\begin{biblist}

\bib{Ananyan-Hochster}{article}{
   author={Ananyan, Tigran},
   author={Hochster, Melvin},
   title={Ideals generated by quadratic polynomials},
   series={Math. Research Letters},
   volume={19},
   publisher={Math. Research Letters},
   date={2012},
   pages={233--244},
}

\bib{Ananyan-Hochster-small}{article}{
   author={Ananyan, Tigran},
   author={Hochster, Melvin},
   title={Small subalgebras of polynomial rings and Stillman's conjecture},
   series ={arXiv:1610.09268},
   publisher={preprint},
   date={2016},
}

\bib{Ananyan-Hochster-strength}{article}{
   author={Ananyan, Tigran},
   author={Hochster, Melvin},
   title={Strength conditions, small subalgebras, and Stillman bounds in degree},
   series={arXiv:1810.00413},
   publisher={Math. AC},
   date={2018},
}

\bib{Erman-Sam-Snowden}{article}{
   author={Erman, Daniel},
   author={Sam, Steven},
   author = {Snowden, Andrew},
   title={Big polynomial rings and Stillman's conjecture},
   series ={arXiv:1801.09852v2},
   publisher={math. AC},
   date={2018},
}

\bib{Miller-Stephenson}{article}{
   author={Miller, Nina},
   author={Stephenson, Darin},
   author = {Wells, Andrew},
   author = {Erika, Wittenborn},
   title={Counting Quadratic Forms of Rank 1 and 2},
   publisher={Pi Mu Epsilon Journal Vol. 12},
   date={2009},
}

\bib{Eisenbud}{book}{
   author={Eisenbud, David},
   title={Commutative Algebra with a view toward Algebraic Geometry},
   series ={Graduate Texts in Mathematics},
   publisher={Springer},
   date={2004},
}

\bib{Harris}{book}{
   author={Harris, Joe},
   title={Algebraic Geometry},
   series ={Graduate Texts in Mathematics},
   publisher={Springer},
   date={1992},
}


\end{biblist}
\end{bibdiv}
\end{document}